\documentclass[12pt,reqno]{amsart}

\usepackage[mathlines]{lineno}

\usepackage{times}
\usepackage{amsfonts}
\usepackage{amssymb}
\usepackage{latexsym}
\usepackage{xcolor}

\usepackage{comment}

\newtheorem{theorem}{Theorem}
\newtheorem{lemma}[theorem]{Lemma}

\newtheorem{definition}[theorem]{Definition}
\newtheorem{remark}{Remark}
\newtheorem{corollary}[theorem]{Corollary}
\newtheorem{example}[theorem]{Example}

\def\nin{\relax\hbox{$/\kern-.7em{\rm \in\,}$}}

\usepackage{comment}

\begin{document}

%%%%%%%%%%%%%%%%%%%%%%%%%%%%%%%%%%%%%%%
%%%%%%%%%%%%%%%%%%%%%%%%%%%%%%%%%%%%%%%
%%%%%%%%%%%%%%%%%%%%%%%%%%%%%%%%%%%%%%%
%%%%%%%%%%%%%%%%%%%%%%%%%%%%%%%%%%%%%%%

\centerline {\Large Boundedness of  Hausdorff-type operators}
\centerline {\Large with  two-variable kernels on Lebesgue spaces}\label{ch:6}

\

\centerline { A. R. Mirotin}

\centerline {amirotin@yandex.ru}

\bigskip

Abstract.  A general  concept of a Hausdorff-type operator that absorbs  all types of operators bearing the name `` Hausdorff operator'' and many others is considered. The characteristic features of this  concept are the consideration of kernels depending on an external variable and the action between two arbitrary different sets.
  Generalizations and analogs of  classical results on $L^p$ boundedness  of various type of Hausdorff operators  are proved for the case of such operators.
\bigskip

Key words and phrases.  Hausdorff operator, Integral operator, Two-variable kernel,  Lebesgue space, Bounded operator

\bigskip

\section{Introduction and preliminaries}\label{sec:intr_two}

In resent two  decades many different notions of a Hausdorff operator have been suggested (see, e.~g., \cite{BM1, BM2, CGS,  CFW, DK-2024,  DZDW, GS, KM,  KM2, KSZhu, LL, emj, JMAA,  Lie, MCAOT, MCAOT2, S,  SG}, and especially the survey article \cite{LM:survey}).

In this note following \cite{Mir_Lobach} we consider  a general  concept of a Hausdorff-type operator  that absorbs  all types of operators considered in the works mentioned above and many others.  The characteristic features of our definition are the consideration of kernels depending on an external variable and the action between two arbitrary different sets.
  Generalizations and analogs of  classical results on $L^p$ boundedness  of various type of Hausdorff operators  are proved for the case of such operators.
  
  The note  consists of two sections: Introduction  and preliminaries and  boundedness  of a  Hausdorff-type operators with a two-variable kernel  on Lebesgue spaces.

Throughout  we will accept the next definition.

\begin{definition}\label{df:general1_two}  Let  $S$ and  $S'$ be two sets, $(\Omega,\mu)$ denotes some measure space, and $A(u): S\to S'$ ($u\in  \Omega $) be some family of mappings\footnote{$A(u)$ do not assumed to be invertible.}.  Let  $\Phi(u,x)$ be a given  function on $\Omega\times S$ which is $\mu$-measurable for every fixed $x\in S$, and $V$ some Banach space.
 {\it A Hausdorff-type operator  with a two-variable kernel} acts on a  functions $f: S' \to V$  by the rule
\begin{equation}\label{general_two}
(\mathcal{H}_{\Phi,A,\mu}f)(x) =\int_{\Omega} \Phi(u,x)f(A(u)(x))d\mu(u), \ x\in S,
\end{equation}
%\end{comment}
provided the integral  converges in a suitable  sense.
\end{definition}
In this note, we consider the case  $V=\mathbb{C}$ only.

As was mentioned above Hausdorff operators in such generality were first considered in \cite{Mir_Lobach}.
 The Definition \ref{df:general1_two}  covers all known classes of operators bearing the name ``Hausdorff'' and also many other classical and new types of operators and transformations such as  transformations arise in  classical summation methods,  classical and  discrete Hilbert transforms and there generalizations, integral Hankel operators, orbital integrals,  convolution operators on groups, Hadamard-Bergmann convolutions etc. 
  (see \cite{LM:survey, arxGeneral})

In the case were the kernel $\Phi(u,x)=\Phi(u)$ does not depend of $x$ we call $\mathcal{H}_{\Phi,A,\mu}$ {\it a Hausdorff-type operator  with a one-variable kernel.} Thus, a Hausdorff-type operator  with a one-variable kernel looks as follows

\begin{equation}\label{general_one}
(\mathcal{H}_{\Phi,A,\mu}f)(x) =\int_{\Omega} \Phi(u)f(A(u)(x))d\mu(u), \ x\in S.
\end{equation}

Of cause, in the case where $A(u)(x)=u$ for all $x$ and $u$, the Definition \ref{df:general1_two}  turns into the classical definition of an integral operator. But we  are interested in the opposite ``nonclassical'' case. This class of operators bears specific features comparison with the classical one. For example, it contains an important  subclass of operators with a one-variable kernel,  a non-trivial one-dimensional Hausdorff operator  in $L^p$ with a one-variable kernel is always non-compact,  non-quasinilpotent, and even non Riesz (see \cite{RJMP}).  Even in the case of a one-variable kernel this   class  covers a lot of classical and new types of operators but, e.~g., Hadamard-Bergman  operators with two-variable kernels  fit into our definition, too (see \cite{KMirM, arxGeneral}).

\bigskip

\section{Boundedness  of a  Hausdorff-type operators with a two-variable kernel on Lebesgue spaces}\label{sec:Lqp_two}

\bigskip

In this section, $(\Omega,\mu)$, $(S,\nu)$,  and $(S',\nu')$ denote three  measure spaces. We will assume the following  rather weak form of interaction of the family of mappings $A(u)$ with measures $\nu$ and $\nu'$.

\begin{definition}\label{egree} %Let $(S, \mathcal{B}, \nu)$ be a   measure space.
 We say that a family $A(u): S\to S'$ ($u\in \Omega$) 
\textit{agrees with  measures $\nu$ and  $\nu'$ in a weak form}   if for each $\nu'$-measurable set $E\subset S'$ of finite measure and for every $u\in \Omega$ the set $A(u)^{-1}(E)$ is  $\nu$-measurable and
\begin{equation}\label{nuA(u)}
\nu(A(u)^{-1}(E))\le m(u)^{-1}\nu'(E)
\end{equation}
for some positive $\mu$-measurable function $m$ on $\Omega$.\footnote{In fact, $m(u)$  depends on $A(u)$.}

In the case where $(S,\nu)=(S',\nu')$ we say that $(A(u))_{u\in \Omega}$
\textit{agrees with   $\nu$  in a weak form}.
\end{definition}

\begin{lemma}\label{lm:IfA(u)}
Under the assumptions of  Definition \ref{egree} one has  for a function  $g\in L^1(\nu')$, $g(x')\ge 0$  that
\begin{equation}\label{IfA(u)}
\int_{S}g(A(u)(x))d\nu(x)\le m(u)^{-1}\int_{S'}g(x')d\nu'(x)
\end{equation}
\end{lemma}
\begin{proof}
If $g=\chi_{E}$ where $\nu'(E)<\infty$ the inequality \eqref{IfA(u)} turns into \eqref{nuA(u)}  ($\chi_{E}$ denotes the indicator of a  set $E$). So,    the inequality \eqref{IfA(u)} is valid for all simple functions $g=\sum_j c_j\chi_{E_j}$ with $c_j\ge 0$, $\nu'(E_j)<\infty$. In the general case, one can choose a sequence of non-negative simple functions $g_n$ such that $g_n\uparrow g$  point-wise ($n\to\infty$) and apply the Monotone Convergence Theorem.
\end{proof}

For the proof of the main result of this note, we put for an arbitrary real $p$ and $q$
\begin{equation}\label{r}
r=r(p,q):=
\begin{cases}
\frac{q}{q-p}, \mbox{ if } p\ne q,\\
\infty, \mbox{ if } p=q,
\end{cases} \ r(p,\infty):=1,
\end{equation}
and for  a given  function $\Phi(u,x)$ on $\Omega\times S$ which is $\mu\otimes\nu$-measurable we introduce the mixed norms
$$
\|\Phi\|_{\mu,\nu,m}^{(p,q)}:=\int_\Omega\|\Phi(u,\cdot)\|_{L^{pr}(\nu)}m(u)^{-\frac 1q}d\mu(u),
$$
$$
\|\Phi\|_{\mu,\nu,m}^{(p,\infty)}:=\int_\Omega\|\Phi(u,\cdot)\|_{L^{p}(\nu)}d\mu(u),
$$
$$
\|\Phi\|_{\mu,\nu,m}^{(\infty,\infty)}:={\rm esssup}_{x\in S}\int_\Omega |\Phi(u,x)|d\mu(u).
$$
In particular if $p\ne \infty$
$$
\|\Phi\|_{\mu,\nu,m}^{(p,p)}:=\int_\Omega \|\Phi(u,\cdot)\|_{L^{\infty}(\nu)}m(u)^{-\frac 1p}d\mu(u).
$$

\begin{theorem}\label{th:bdd_Lqp_two} Let $\infty>q\ge p\ge 1$,  or  $\infty=q> p\ge 1$, or $q=p= \infty$, the measure $\nu$ be   sigma-finite, and the family $(A(u))_{u\in \Omega}$ agrees with $\nu$ and $\nu'$  in a weak form.
If
\begin{equation}\label{Phi_qp}
 \|\Phi\|_{\mu,\nu,m}^{(p,q)}<\infty
\end{equation}
then the operator \eqref{df:general1_two}
 is bounded as an operator between $L^q(\nu')$ and  $L^p(\nu)$, and its norm does not exceed $\|\Phi\|_{\mu,\nu,m}^{(p,q)}$.
\end{theorem}

\begin{proof}
Let $p\ne \infty$. Using Minkowskii integral
inequality  we have for $1\le p<\infty$ and for an arbitrary $f\in L^q(\nu')$
\begin{eqnarray}\label{2_two}
\|\mathcal{H}_{\Phi, A,\mu}f\|_{L^p(\nu)}&=&\left(\int_{S} \left|\int_\Omega \Phi(u,x)f(A(u)(x))d\mu(u) \right|^pd\nu(x)\right)^{\frac 1p}\\ \nonumber
&\leq&
\int_\Omega\left(\int_{S}|\Phi(u,x)|^p|f(A(u)(x))|^pd\nu(x)\right)^{\frac 1p}d\mu(u).
\end{eqnarray}

First we assume that  $\infty>q\ge p\ge 1$.

Then by the Holder inequality and Lemma \ref{lm:IfA(u)}
\begin{align}\label{3_two}
\int_{S}|\Phi(u,x)|^p|f(A(u)(x))|^pd\nu(x)\\  \nonumber
&\hspace{-28mm}\le\left(\int_{S}|\Phi(u,x)|^{pr}d\nu(x)\right)^{\frac 1r} \left(\int_{S}|f(A(u)(x))|^{pr'}d\nu(x)\right)^{\frac 1{r'}}\\  \nonumber
&\hspace{-28mm}\le m(u)^{-\frac 1{r'}}\left(\int_{S}|\Phi(u,x)|^{pr}d\nu(x)\right)^{\frac 1r}
\left(\int_{S}|fx')|^{pr'}d\nu'(x')\right)^{\frac 1{r'}}
\end{align}
where $r$ is given by formula \eqref{r} and, as usual, $\frac 1r+\frac 1{r'}=1$.  Substituting this into \eqref{2_two} and taking into account that $pr'=q$ we have
\begin{eqnarray*}
\|\mathcal{H}_{\Phi, A,\mu}f\|_{L^p(\nu)}
\leq  \|\Phi\|_{\mu,\nu,m}^{(p,q)}\|f\|_{L^q(\nu')}
\end{eqnarray*}
and thus
\begin{eqnarray}\label{4_two}
\|\mathcal{H}_{\Phi, A,\mu}\|_{L^q(\nu')\to L^p(\nu)}
\leq  \|\Phi\|_{\mu,\nu,m}^{(p,q)}.
\end{eqnarray}

If $\infty=q> p\ge 1$, then in the previous argument the inequality \eqref{3_two} should be replaced by
\begin{align}\label{10_two}
\int_{S}|\Phi(u,x)|^p|f(A(u)(x))|^pd\nu(x)
\le\int_{S}|\Phi(u,x)|^{p}d\nu(x)\|f\|_{L^\infty(\nu')}^p
\end{align}
which is valid for every $u\in \Omega$ due to \eqref{nuA(u)}. Indeed, let 
$$
N':=\{x'\in S': |f(x')|>\|f\|_{L^\infty(\nu')}\},
$$
and
$$
N_u:=\{x\in S: |f(A(u)(x))|>\|f\|_{L^\infty(\nu')}\}.
$$
Then $\nu'(N')=0$, $N_u\subseteq A(u)^{-1}(N')$ and  therefore $\nu(N_u)=0$ for every $u\in \Omega$ by \eqref{nuA(u)}.
Thus, for every $u\in \Omega$ 
\begin{align}\label{20_two}
 |f(A(u)(x))|\le \|f\|_{L^\infty(\nu')}^p
 \end{align}
  for $\nu$-a.~e. $x\in S$
and therefore \eqref{10_two} holds for all $u\in \Omega$.

In view of \eqref{20_two} the proof in the case $q=p= \infty$ is almost obvious:
\begin{eqnarray*}
\|\mathcal{H}_{\Phi, A,\mu}f\|_{L^\infty(\nu)}&\leq& {\rm esssup}_{x\in S}\int_\Omega |\Phi(u,x)||f(A(u)(x))|d\mu(u)\\
&\leq & {\rm esssup}_{x\in S}\int_\Omega |\Phi(u,x)|d\mu(u)\|f\|_{L^\infty(\nu')}\\
&=&\|\Phi\|_{\mu,\nu,m}^{(\infty,\infty)}\|f\|_{L^\infty(\nu')}.
\end{eqnarray*}
\end{proof}

In the next corollary for a $\mu$-measurable function $\varphi$ on $\Omega$ we put
$$ 
  \|\varphi\|_{m,\mu}^{(p)}:=\int_\Omega |\varphi(u)|m(u)^{-\frac 1p}d\mu(u)
$$
(here $ \|\varphi\|_{m,\mu}^{\infty}:=\|\varphi\|_{L^1(\mu)}$).

\begin{corollary}\label{cr:Lp} Let   the measure $\nu$ be   sigma-finite, the family $(A(u))_{u\in \Omega}$ agrees with $\nu$ and $\nu'$  in a weak form,  and $1\leq p< \infty$. Let  a $\mu$-measurable function $\varphi$ on $\Omega$ be such that $|\Phi(u,x)|\le \varphi(u)$ for $\mu$-a.~e. $u\in  \Omega$ and all $x\in S$. If $\|\varphi\|_{m,\mu}^{(p)}<\infty$
 then the operator \eqref{general_one}
 is bounded as an operator between $L^p(\nu')$ and  $L^p(\nu)$, and its norm does not exceed $ \|\varphi\|_{m,\mu}^{(p)}$.
\end{corollary}

\begin{proof} This follow from Theorem \ref{th:bdd_Lqp_two} with $p=q$.
 \end{proof}

As a special case of the  previous result we obtain the following

\begin{corollary}\label{cr:Lp2} Let  the measure $\nu$ be   sigma-finite, the family $(A(u))_{u\in \Omega}$ agrees with $\nu$  and  $\nu'$   in a weak form,  and $1\leq p< \infty$.  If a $\mu$-measurable function $\Phi(u)$ on $\Omega$ satisfies $\|\Phi\|_{m,\mu}^{(p)}<\infty$
 then a Hausdorff-type  operator \eqref{general_one}  with a one variable kernel  is bounded as an operator between $L^p(\nu')$ and  $L^p(\nu)$, and its norm does not exceed $ \|\Phi\|_{m,\mu}^{(p)}$.
\end{corollary}

In connection with the next corollary, we note that throughout 
 ${\rm mod}(\psi)$ stands for   the modulus of a topological automorphism $\psi$ of a locally compact group $G$, and
 the family $\mathrm{Aut}(G)$ of all topological automorphisms of  $G$ is assumed to be equipped with its natural topology (the Braconnier topology).
In this topology the sets 
$$
\mathcal{O}(C, W):=\{\psi\in \mathrm{Aut}(G): \psi(x)x^{-1}\in W,   \psi^{-1}(x)x^{-1}\in W \forall x\in C\} 
$$
where     $C$ runs over  all
compact subsets of $G$ and $W$ runs over all neighborhoods of the unit  $e\in G$
constitute a fundamental system of neighborhoods of the identity  (see, e.~g., \cite[(26.1)]{HiR}, \cite[Section III.3]{Hoch}).

\begin{corollary}\label{cr:Lp_group} Let   $G$ be a $\sigma$-compact locally compact group     with   the Haar measure $\nu$.    Assume that the map $u\mapsto  A(u)$ from $\Omega$ to $\mathrm{Aut}(G)$ is measurable with respect to the measure $\mu$ in $\Omega$ and a  function $\Phi(u,x)$ on $\Omega\times S$ is $\mu\otimes\nu$-measurable. Then the statement of Theorem \ref{th:bdd_Lqp_two} is valid with  $S=S'=G$  and $m(u)={\rm mod}(A(u))$  the modulus of a topological automorphism $A(u)$.
\end{corollary}

\begin{proof} 
  In this case all the conditions of Definition \ref{egree}  with $S=S'=G$    are fulfilled for $(A(u))_{u\in \Omega}$.  Indeed, if we take  $m(u)={\rm mod}(A(u))$,   the equalitis in    \eqref{nuA(u)} and \eqref{IfA(u)}  are valid (see, e.~g., \cite{HiR}).  Since  the map $\psi\mapsto {\rm mod}(\psi)$ from  $\mathrm{Aut}(G)$ to $(0,\infty)$ is continuous (see, e.~g.,  \cite[(26.21)]{HiR}),  the family $(A(u))_{u\in \Omega}$  agrees with the measure $\nu$, and Theorem \ref{th:bdd_Lqp_two} is applicable with   $S=S'=G$. 
\end{proof}

The next corollary shows in particular that the condition in  Corollary \ref{cr:Lp} that  $\|\Phi\|_{m,\mu}^{(p)}<\infty$    can not be weakened in general.  This implies that the condition \eqref{Phi_qp} in Theorem  \ref{th:bdd_Lqp_two}   can not be weakened in general, too. It shows also that the equality is possible in \eqref{4_two} (see also Remark \ref{discr} below).

\begin{corollary}\label{cr:Lp_compactgroup} Let $S=S'=G$ be a compact or discrete  group with   the normalized Haar measure $\nu$, and  $\Phi(u,x)= \Phi(u)$ a $\mu$-measurable one variable kernel.  Assume that the map $u\mapsto  A(u)$ from $\Omega$ to $\mathrm{Aut}(G)$ is measurable with respect to the measure $\mu$ in $\Omega$ and  $1\leq p\le \infty$. 

(i) If $\Phi\in L^1(\mu)$ then the operator \eqref{general_one}
 is bounded in $L^p(\nu)$ and  one has $\|\mathcal{H}_{\Phi, A,\mu}\|_{L^p\to L^p}$$\le \|\Phi\|_{L^1(\mu)}$.
 
 (ii) If $G$ is  compact and $\Phi(u)\ge 0$ then the condition $\Phi\in L^1(\mu)$  is also necessary for the boundedness of   the operator \eqref{general_one} in $L^p(\nu)$,  and one has $\|\mathcal{H}_{\Phi, A,\mu}\|_{L^p\to L^p}$$=\|\Phi\|_{ L^1(\mu)}$.
 \end{corollary}
 
 \begin{proof} 
 (i) This follows from Corollary \ref{cr:Lp_group}, since ${\rm mod}(A(u))\equiv 1$ for compact or discrete  group $G$.
 
 (ii) 
  If  the operator $\mathcal{H}_{\Phi, A,\mu}$
 is bounded in $L^p(\nu)$ then for the function $f\equiv 1$ on $G$ we have $\|f\|_{L^p(\nu)}=1$, and $\mathcal{H}_{\Phi, A,\mu}f\in L^p(\nu)$.  Therefore
 \begin{align*}
 \infty&>\|\mathcal{H}_{\Phi, A,\mu}f\|_{L^p(\nu)}=\left(\int_G |\mathcal{H}_{\Phi, A,\mu}f(x)|^pd\nu(x)\right)^{\frac1p}\\
 &=\left(\int_G \left(\int_\Omega \Phi(u)d\mu(u)\right)^pd\nu(x)\right)^{\frac1p}\\
 &= \int_\Omega \Phi(u)d\mu(u)=\|\Phi\|_{ L^1(\mu)}.
 \end{align*}
 Thus, $\Phi\in L^1(\mu)$   and $\|\mathcal{H}_{\Phi, A,\mu}\|_{L^p\to L^p}\ge\|\Phi\|_{ L^1(\mu)}$.
\end{proof}

\begin{remark}\label{discr} Let    $1\le p\le \infty$,  $S=S'$ be the additive group $\mathbb{R}^d$, $\Omega=\mathbb{Z}$ endowed with the counting measure, $\varphi:\mathbb{Z}\to \mathbb{C}$. The general form of a topological   automorphism of   $\mathbb{R}^d$  looks as $\psi(x)= Ax$, where $A\in {\rm GL}(d,\mathbb{R})$. If  $A_k\in {\rm GL}(d,\mathbb{R})$ an operator \eqref{general_one} turns into discrete Hausdorff operator of the form
$$
(\mathcal{H}_{\varphi, A}f)(x)=\sum_{k\in \mathbb{Z}}\varphi(k)f(A_kx).
$$
 In this case, $m(k)=|\det A_k|$  and therefore
$$
\|\varphi\|_{m,\mu}^{(p)}=\sum_{k\in \mathbb{Z}}|\varphi(k)||\det A_k|^{-\frac 1p}.
$$
As was shown in \cite[Lemma 2.2]{MMAS} the condition $\|\varphi\|_{m,\mu}^{(p)}<\infty$ in the Corollary \ref{cr:Lp} for the $L^p$ boundedness of $\mathcal{H}_{\varphi,A}$ cannot be weakened in general,
	and the equality $\|\mathcal{H}_{\varphi, A}\|_{L^p\to L^p}=\|\varphi\|_{m,\mu}^{(p)}$  is possible.
\end{remark}

\begin{example}
One cannot expect  that Hausdorff operator is bounded in $L^p$ for $0<p<1$ because even the simplest nontrivial Hausdorff operator, namely
 the continuous Ces\'{a}ro operator
$$
(\mathcal{C}f)(x)=\int_0^1f(ux)du
$$
is unbounded in $L^p(0,1)$ for $0<p<1$. Indeed, consider the function $f_\alpha(x)=x^{-\alpha}$ where $1<\alpha<\frac 1p$. Then $f_\alpha\in L^p(0,1)\setminus L^1(0,1)$ and $(\mathcal{C}f_\alpha)(x)=\infty$.\footnote{On the other hand, a wide class of  Hausdorff operators is bounded in Hardy spaces $H^p$ for $0<p<1$ (see \cite{LiMi1}).}
\end{example}

\end{document}